\documentclass[10pt]{amsart}
\usepackage{amsmath,amssymb,amsthm,a4wide}
\usepackage[all]{xy}
\usepackage{graphicx}
\usepackage{epstopdf}
\newtheorem{theorem}{Theorem}[section]    
\newtheorem{lemma}[theorem]{Lemma}          
\newtheorem{proposition}[theorem]{Proposition}  
\newtheorem{prop-definition}[theorem]{Proposition-Definition}  

\theoremstyle{definition}
\newtheorem{definition}[theorem]{Definition}
\newtheorem{remark}[theorem]{Remark}

\newtheorem{example}[theorem]{Example}

 \makeatletter
    
    \@addtoreset{equation}{section}
  \makeatother
\def\co{\colon \thinspace}

\newcommand{\GL}{\textrm{GL}}
\newcommand{\Z}{\mathbb{Z}}

\newcommand{\C}{\mathbb{C}}
\newcommand{\sltwo}{\mathfrak{sl}_{2}}

\newcommand{\id}{\mathsf{id}}
\newcommand{\bF}{\mathbb{F}}
\newcommand{\be}{\mathbf{e}}
\newcommand{\half}{\frac{1}{2}}

\DeclareMathOperator{\trace}{trace}

\DeclareMathOperator{\End}{End}

\DeclareMathOperator{\Ker}{Ker}

\title[Colored Alexander invariants]{A homological representation formula of colored Alexander invariants}
\author{Tetsuya Ito}
\address{Research Institute for Mathematical Sciences, Kyoto university, Kyoto, 606-8502, Japan}
\email{tetitoh@kurims.kyoto-u.ac.jp}
\urladdr{http://www.kurims.kyoto-u.ac.jp/~tetitoh/}
\subjclass[2010]{Primary~57M27 
, Secondary~57M25,20F36}
\keywords{Alexander polynomial, Colored Alexander invariant, (truncated) Lawrence's representation}

\begin{document}

\begin{abstract} 
We give a formula of the colored Alexander invariant in terms of the homological representation of the braid groups which we call truncated Lawrence's representation. This formula generalizes the famous Burau representation formula of the Alexander polynomial.
\end{abstract}
\maketitle

\section{Introduction}

The Alexander polynomial is one of the most important and fundamental knot invariant having various definitions and interpretations. Each definition brings a different prospect and often leads to different generalizations.

In this paper we explore one of the generalizations of the Alexander polynomial, the \emph{colored Alexander invariant} introduced in \cite{ado}. This is a family of invariants $\Phi^{N}_{K}$ of a link $K$ indexed by positive integers $N=2,3,\ldots$, and $\Phi_{K}^{2}$ coincides with the (multivariable) Alexander polynomial \cite{mu1}. The first definition of the colored Alexander invariant in \cite{ado} uses a state-sum inspired from solvable vertex model in physics. As is already noted in \cite{ado} and made it clarified in \cite{mu2}, the $N$-th colored Alexander invariant $\Phi^{N}_{K}$ is a version of a quantum $\sltwo$ invariant at $2N$-th root of unity.

Throughout the paper we treat the case that $K$ is a knot. Then the colored Alexander invariant $\Phi^{N}_{K}(\lambda)$ is a function of one variable $\lambda$. We give a new formulation of the colored Alexander invariant, in the same spirit as \cite{i2} where we gave a topological formulation of the loop expansion of the colored Jones polynomials.
 
In Theorem \ref{theorem:main} we show that the colored Alexander invariant is written as a sum of the traces of \emph{homological} representations which we call \emph{truncated Lawrence's representation}. These are variants of Lawrence's representation studied in \cite{law,i1,i2}, derived from an action on the configuration space. Our formula can be seen as a generalization of the Burau representation formula of the Alexander polynomial and has more topological flavor.

A key result is Theorem \ref{theorem:key}, where we show that truncated Lawrence's representation is identified with a certain quantum $\sltwo$ braid group representation, defined for the case the quantum parameter $q$ is put as $2N$-th root of unity. This generalizes a relation between Lawrence's representation and generic quantum $\sltwo$-representation \cite{koh,i1}, and is interesting in its own right.

Unfortunately, unlike the Burau representation formula, our formula is not completely topological since it heavily depends on a particular presentation (closed braid representative) of knots, and we cannot see its topological invariance directly, although it suggests a relationship to the topology of abelian coverings of the configuration space.

We remark that the colored Alexander invariant $\Phi^{N}_{K}(\lambda)$ at $\lambda=(N-1)$ is equal to the $N$-colored Jones polynomial at $N$-th root of unity \cite{mumu}, which in turn, is equal to Kashaev's invariant derived from the quantum dilogarithm function \cite{kas}. Thus, our formula also brings a new prospect for Kashaev's invariant and the volume conjecture.

\section*{Acknowledgements}
The author gratefully thanks Jun Murakami who suggests a generalization of the author's previous work on colored Jones polynomial to the colored Alexander invariants. He also wish to thank Tomotada Ohtsuki for stimulating discussion.
The author was supported by JSPS KAKENHI Grant Numbers 25887030 and 15K17540.

\section{Quantum $\sltwo$ representation and colored Alexander invariant}
\label{section:q}

\subsection{Generic quantum $\sltwo$ representation}

For a complex parameter $q \in \C-\{0,1\}$, we define 
\[ [a]_{q}= \frac{q^{a}-q^{-a}}{q-q^{-1}},\quad [a;i]_{q} = \begin{cases} 1 & (i=0), \\
[a]_{q} \cdots [a+1-i]_{q} & (i>0).
\end{cases}
 \]
\[ [a]_{q}! = [a;a]_{q} = [a]_{q}[a-1]_{q}\cdots[1]_{q}, \quad {a \brack b}_{q
}= \frac{[a]_{q}!}{[b]_{q}![a-b]_{q}!}.\]

Let $U_{q}(\sltwo)$ be the quantum $\sltwo$, defined by
\[ U_{q}(\sltwo) = \left\langle K,K^{-1},E,F \: 
\begin{array}{|cc} KK^{-1}=K^{-1}K=1, & [E,F] = \frac{K-K^{-1}}{q-q^{-1}}\\ KEK^{-1}=q^{2}E, & KFK^{-1}=q^{-2}F \end{array} \right\rangle.\]

For $\lambda \in \C^{*}$, let $\widehat{V_{\lambda}}$ be a $\C$-vector space spanned by $\{\widehat{v}^{\lambda}_{0},\widehat{v}^{\lambda}_{1},\ldots,\}$, equipped with an $U_{q}(\sltwo)$-module structure by
\begin{gather*}
\begin{cases}
 K \widehat{v}^{\lambda}_{i} = q^{(\lambda -2i)} \widehat{v}^{\lambda}_{i} \\
 E \widehat{v}^{\lambda}_{i} = \widehat{v}^{\lambda}_{i-1} \\
 F \widehat{v}^{\lambda}_{i} = [i+1]_{q}[\lambda-i]_{q}\widehat{v}^{\lambda}_{i+1}.
\end{cases}
\end{gather*}
Here $[\lambda-i]_{q} = \frac{q^{\lambda-i}-q^{-(\lambda-i)}}{q-q^{-1}}$. 
We define the $R$-operator $\mathsf{R}$ by
\[ \mathsf{R} = \mathsf{R_{\lambda,\mu}} = q^{-\half \lambda\mu} T \circ \mathcal{R} : \widehat{V_{\lambda}} \otimes \widehat{V_{\mu}} \rightarrow \widehat{V_{\mu}} \otimes \widehat{V_{\lambda}}, \]
where $T$ is the transposition $T(v\otimes w) = w\otimes v$ and $\mathcal{R}$ is a universal $R$-matrix of $U_{q}(\sltwo)$.
When $\mu=\lambda$, $\mathsf{R}$ gives rise to a braid group representation 
\[ \varphi_{\widehat{V_{\lambda}}}: B_{n} \rightarrow \GL(\widehat{V_{\lambda}}{}^{\otimes n}), \quad \varphi(\sigma_{i}) = \id^{\otimes (i-1)} \otimes \mathsf{R} \otimes \id^{\otimes (n-1-i)}. \]

Let $\widehat{V_{n,m}}=\widehat{V_{n,m}}(\lambda)$ be a subspace of $\widehat{V_{\lambda}}{}^{\otimes n}$ spanned by $\{ \widehat{v}^{\lambda}_{e_1}\otimes \cdots \otimes \widehat{v}^{\lambda}_{e_n}\: | \: e_1+\cdots+e_n=m\}$, and let $\widehat{W_{n,m}} = \widehat{W_{n,m}}(\lambda) = \Ker E \cap \widehat{V_{n,m}}$. 
Then $\widehat{V_{n,m}}$ and $\widehat{W_{n,m}}$ are finite-dimensional braid group representations with dimension $\binom{n+m-1}{m}$ and $\binom{n+m-2}{m}$, respectively.

To relate $\widehat{V_{\lambda}}$ with irreducible finite dimensional $U_{q}(\sltwo)$-modules, let $V_{\lambda}$ be subspace of $\widehat{V_{\lambda}}$ spanned by $\{ v^{\lambda}_{0}, v^{\lambda}_{1}\ldots\}$, where we define $v_{j}^{\lambda}= [\lambda;j]_{q} \widehat{v}^{\lambda}_{j}$. Then $V_{\lambda}$ is a sub $U_{q}(\sltwo)$-module with much familiar action
\begin{gather}
\label{eqn:verma0}
\begin{cases}
 K v^{\lambda}_{i} = q^{(\lambda -2i)} v^{\lambda}_{i} \\
 E v^{\lambda}_{i} = [\lambda+1-i]_{q}v^{\lambda}_{i-1} \\
 F v^{\lambda}_{i} = [i+1]_{q}v^{\lambda}_{i+1}.
\end{cases}
\end{gather}
If $\lambda$ is equal to a positive integer $\alpha-1$, then $v^{\lambda}_{i}=0$ for $i\geq \alpha$, and $V_{\lambda}$ is nothing but the $\alpha$-dimensional irreducible $U_{q}(\sltwo)$-module. 

As in the case $\widehat{V_{\lambda}}$, $\mathsf{R}$ leads to the braid group representation $\varphi_{V_{\lambda}}: B_{n} \rightarrow \GL(V_{\lambda}^{\otimes n})$,
and $V_{n,m} = \widehat{V_{n,m}} \cap V_{\lambda}^{\otimes n}$ and $W_{n,m}= \widehat{W_{n,m}} \cap V_{\lambda}^{\otimes n}$ also give rise to braid group representations.

In the definition of $\mathsf{R}$, the term $q^{-\half \lambda \mu}$ corresponds to a part of the framing correction (see also Remark \ref{remark:framing}). Thanks to this modification, the action of $\mathsf{R}_{\lambda,\lambda}$ is written by
\begin{equation}
\label{eqn:raction}
\mathsf{R}(\widehat{v}^{\lambda}_{i} \otimes \widehat{v}^{\lambda}_{j}) = 
 q^{-\lambda(i+j)}\sum_{n=0}^{i}q^{2(i-n)(j+n)}q^{\frac{n(n-1)}{2}}{ n+j\brack j}_{q}[\lambda-i;n]_{q}(q-q^{-1})^{n}\, \widehat{v}^{\lambda}_{j+n} \otimes \widehat{v}^{\lambda}_{i-n}.
\end{equation}

A remarkable feature is that each coefficient in (\ref{eqn:raction}) is a Laurent polynomial of $q$ and $q^{\lambda}$. Thus by regarding $q$ and $q^{\lambda}$ as abstract variables, one can define the braid group representation $\widehat{V_{\lambda}}{}^{\otimes n}$ over the Laurent polynomial ring $\Z[q^{\pm 1}, q^{\pm \lambda}]$, which we call a \emph{generic quantum $\sltwo$-representation} (see \cite{jk,i1} for details). 

We will consider the following two kind of genericity conditions.
\begin{definition}
We say that \emph{$\lambda$ is generic with respect to $q$} if $
[\lambda-i]_{q}\neq 0 \text{ for all } i \in \Z$. 
We also say that \emph{$q$ and $\lambda$ are fully generic} if the subset $\{1,q,q^{\lambda}\}$ is algebraically independent.
\end{definition}

\begin{lemma}
\label{lemma:generic}
{$ $}
\begin{enumerate}
\item[(i)] If $\lambda$ is generic with respect to $q$, then $\widehat{V_{\lambda}}$ and $V_{\lambda}$ are isomorphic as $U_{q}(\sltwo)$-modules. In particular, the braid group representations $\widehat{V_{\lambda}}{}^{\otimes n}$ and $\widehat{V_{\lambda}}$, ($\widehat{V_{n,m}}$ and $V_{n,m}$, $\widehat{W_{n,m}}$ and $W_{n,m}$) are isomorphic.
\item[(ii)] If $\lambda$ and $q$ are fully generic, then the braid group representation $\widehat{V_{\lambda}}{}^{\otimes n}$ splits as
\[ \widehat{V_{\lambda}}{}^{\otimes n} \cong \bigoplus_{m=0}^{\infty} \widehat{V_{n,m}} \cong \bigoplus_{m=0}^{\infty} \left(\bigoplus_{k=0}^{m} F^{m-k}\widehat{W_{n,k}} \right). \]
\end{enumerate}
\end{lemma}
\begin{proof}
(i) follows from the definition of $V_{\lambda}$. (ii) is \cite[Lemma 13]{jk}. (We remark that the modules $\widehat{V_{n,m}}$ and $\widehat{W_{n,m}}$ in this paper correspond to $V_{n,m}$ and $W_{n,m}$ in \cite{jk}).
\end{proof}

From now on, we will always assume that $\lambda$ is generic with respect to $q$ so that we do not need to distinguish $\widehat{V_{\lambda}}$ with $V_{\lambda}$. We will mainly work on $V_{\lambda}$.

\subsection{The case $q$ is a root of unity}
Let us put $q$ as the $2N$-th root of unity $\zeta=\zeta_{N}=\exp(\frac{2\pi\sqrt{-1}}{2N})$. As we remarked in the previous section, we assume that $\lambda$ is generic with respect to $\zeta$.

By (\ref{eqn:verma0}), $\{v^{\lambda}_0,\ldots, v^{\lambda}_{N-1}\} \subset V_{\lambda}$ spans an irreducible $N$-dimensional irreducible $U_{\zeta}(\sltwo)$-module, a central deformation of the $N$-dimensional irreducible $U_{q}(\sltwo)$ module \cite{mu2}. 
We denote this $N$-dimensional irreducible $U_{\zeta}(\sltwo)$-module by $U_{N}(\lambda)$, and the corresponding braid group representation by
$ \varphi_{U_{N}}: B_{n} \rightarrow \GL(U_{N}(\lambda)^{\otimes n})$.

As an $U_{\zeta}(\sltwo)$-module, the tensor product decomposes as 
\begin{equation}
\label{eqn:split}
U_{N}(\lambda) \otimes U_{N}(\mu) \cong \bigoplus_{i=0}^{N-1}U_{N}(\lambda +\mu -2i)
\end{equation}
so iterated use of (\ref{eqn:split}) gives a decomposition as an $U_{\zeta}(\sltwo)$ module
\begin{equation}
\label{eqn:split2}
U_{N}(\lambda)^{\otimes n} \cong \bigoplus_{i=0}^{N-1} T_{i} \otimes U_{N}(n \lambda -2i).
\end{equation}
Here $T_{i}$ is the intertwiner space $\End_{U_{\zeta}(\sltwo)}(U_{N}(n\lambda-2i), U_{N}(\lambda)^{\otimes n})$, the set of endomorphisms equivariant with respect to the $U_{\zeta}(\sltwo)$-actions. The dimension of $T_{i}$ is the multiplicity of the direct summands.

The $U_{\zeta}(\sltwo)$-module $U_{N}(\lambda)$ is generated by a highest weight vector $v^{\lambda}_{0}$, so we identify $T_i$ with the subspace consisting of highest weight vectors
\begin{equation}
\label{eqn:split3} T_{i} = \{v \in U_{N}(\lambda)^{\otimes n}\: | \: Ev = 0, Kv = \zeta^{n\lambda - 2i} v\}.
\end{equation}
In particular, we may view the decomposition (\ref{eqn:split2}) as
\begin{equation}
U_{N}(\lambda)^{\otimes n} = \bigoplus_{i=0}^{N-1} \left( T_{i} \oplus F T_{i} \oplus \cdots \oplus F^{N-1}T_{i} \right).
\end{equation}

Next we look at the braid group action. As in the previous section, let $X^{N}_{n,m}=X^{N}_{n,m}(\lambda)$ be the subspace of $U_{N}(\lambda)^{\otimes n}$ spanned by $\{v^{\lambda}_{e_1}\otimes \cdots \otimes v^{\lambda}_{e_n}\: | \: e_1+\cdots+e_n=m\}$, and let $Y^{N}_{n,m} = \Ker E \cap X^{N}_{n,m}$. Both $X^{N}_{n,m}$ and $Y^{N}_{n,m}$ are braid group representations, and as a braid group representation, $U_{N}(\lambda)^{\otimes n}$ splits as 
\[ U_{N}(\lambda)^{\otimes n} \cong \bigoplus_{m=0}^{n(N-1)} X^{N}_{n,m}.\]

By definition, $K v = \zeta^{n\lambda - 2m} v$ for $v \in X^{N}_{n,m}$. Hence by (\ref{eqn:split3}) the intertwiner space $T_i$ in (\ref{eqn:split2}) is identified with the direct sums of $Y_{n,m}^{N}$.
\begin{equation}
\label{eqn:ttoy}
 T_{i} = \{v \in \bigoplus_{m=0}^{n(N-1)} X^{N}_{n,m} \cong U_{N}(\lambda)^{\otimes n}\: | \: Ev=0, \ Kv = \zeta^{\lambda - 2i} v\}  = \bigoplus_{j=0}^{n-1} Y^{N}_{n,i+(N-1)j}.
\end{equation}
Summarizing, the decomposition of $U_{N}(\lambda)^{\otimes n}$ as $U_{\zeta}(\sltwo)$-representation (\ref{eqn:split2}) is written in terms of $Y_{n,m}^{N}$ as 
\begin{equation}
U_{N}(\lambda)^{\otimes n} = \bigoplus_{i=0}^{N-1} \left( \bigoplus_{k=0}^{N-1} F^{k}\left( \bigoplus_{j=0}^{n-1} Y^{N}_{n,i+(N-1)j}\right) \right).
\end{equation}

\subsection{Colored Alexander invariant}

The colored Alexander invariant is the quantum $\sltwo$ invariant coming from $U_{N}(\lambda)$. However, the quantum trace of the quantum representation $\varphi_{U_{N}}$ is trivial so we need a trick (see \cite{gpmt} for a general framework).

For a finite dimensional free $\C$-vector space $V$, $\End(V) \cong V\otimes V^{*}$, where $V^{*}$ is the dual of $V$, and $\trace: \End(V) \rightarrow \C$ is identified with the contraction. 
Let $V^{\otimes n} = V_1 \otimes V_2\otimes \cdots \otimes V_n$ be the tensor products of $n$ copies of $V$.
The \emph{partial trace}
\[ \trace_{2,\ldots,n}: \End(V^{\otimes n}) \rightarrow \End(V) = \End(V_1) \]
is the map defined by the composition of the following natural maps,
\[
\xymatrix{
\End(V^{\otimes n}) \ar[r]^-{\cong} \ar[d]^-{ \trace_{2,\ldots,n}}
& (V^{\otimes n}) \otimes (V^{*}{}^{\otimes n}) \ar[r]^-{\cong}
& (V_{1}\otimes V_{1}^{*})\otimes(V_{2}\otimes V_{2}^{*}) \otimes \cdots \otimes (V_{n}\otimes V_{n}^{*}) \ar[d]_-{\id \otimes \trace \otimes \cdots \otimes \trace}^-{\; (\text{contraction})}\\
\End(V) & \ar@{=}[l]\End(V_1)& \ar[l]^-{\cong}(V_{1}\otimes V_{1}^{*})
}
\]

Let $K$ be an oriented knot in $S^{3}$ represented as the closure of an $n$-braid $\beta$.
We cut the first strand of $\widehat{\beta}$ to get an $(1,1)$-tangle $T_{\beta}$. Let $Q^{U_{N}}(T_{\beta}): U_{N}(\lambda) \rightarrow U_{N}(\lambda)$ be an operator invariant of the tangle $T_{\beta}$ (see \cite[Chapter 3]{oht}), defined by 
\begin{equation}
\label{eqn:cadef}
 Q^{U_{N}}(T_{\beta}) = \zeta^{(N-1)\lambda e(\beta)}\trace_{2,\ldots,n}((\id \otimes h^{\otimes (n-1)})\circ\varphi_{U_N}(\beta)):  U_{N}(\lambda) \rightarrow U_{N}(\lambda) 
\end{equation}
where $h=K^{1-N}: U_{N}(\lambda) \rightarrow U_{N}(\lambda)$ is given by $h(v^{\lambda}_i) = \zeta^{-(N-1)\lambda +2(N-1)i}v^{\lambda}_i$, and $e: B_{n} \rightarrow \Z$ is the exponent sum homomorphism, defined by $e(\sigma_{i}^{\pm 1}) = \pm 1$.

\begin{remark}
\label{remark:framing}
By \cite{mu2} the framing correction is given by $\zeta^{-\half\lambda(\lambda+2-2N)e(\beta)}$, but as we noted in the previous section, the part of the framing correction $\zeta^{-\half \lambda^{2} e(\beta)}$ is already included in the definition of $\mathsf{R}$-operator. This explains the framing correction term $\zeta^{(N-1)\lambda e(\beta)}$ in (\ref{eqn:cadef}).
\end{remark}

Since $U_{N}(\lambda)$ is irreducible, $Q^{U_{N}}(T_{\beta})$ is a scalar multiple of the identity
\begin{equation}
\label{eqn:ca}
Q^{U_{N}}(T_{\beta}) = \Phi_{K}^{N}(\lambda) \id.
\end{equation}
It turns out $\Phi^{N}_{K}(\lambda)$ is independent of a choice of a closed braid representative.

\begin{definition}[Colored Alexander invariant \cite{ado},\cite{mu2}]
The \emph{colored Alexander invariant} of color $N \in \{2,3,\ldots\}$ is the scalar $\Phi_{K}^{N}(\lambda)$ in (\ref{eqn:ca}), viewed as a function of the variable $\lambda$.
\end{definition}

It is convenient to put $x=\zeta^{-2\lambda}$ and write $CA_{K}^{N}(x)=\Phi_{K}^{N}(\lambda)$, a rational function of the variable $x^{\half}$ (Note that by (\ref{eqn:raction}) and (\ref{eqn:cadef}), $CA_{K}^{N}(x)$ is indeed a rational function of $x^{\half}$.)
We call $CA_{K}^{N}(x)$ the \emph{colored Alexander polynomial}, whose name is justified by the fact $CA^{K}_{2}(x) = \Delta_{K}(x)=\nabla_{K}(x^{\frac{1}{2}}-x^{-\frac{1}{2}})$, where $\nabla_{K}(t)$ denotes the Conway polynomial (See \cite{mu1} and Example \ref{exam:N=2}).

\section{Homological representations}
\label{sec:tLaw}

\subsection{Lawrence's representation}

We review a homological braid group representation $L_{n,m}$ which we call {\emph Lawrence's representation}. See \cite{law}, \cite[Section 3]{i1}, \cite[Appendix]{i2}. In the case $m=2$, it is known as \emph{Lawrence-Krammer-Bigelow representation} studied in \cite{kra,big}, and in the case $m=1$ it is nothing but the reduced Burau representation. 

We identify the braid group $B_{n}$ with the mapping class group of the $n$-punctured disc 
\[ D_n=\{|z| \in \C \: | \: |z| \leq n+1 \}-\{1,2,\ldots,n\} = D^{2}-\{p_1,\ldots,p_n \}, \]
so that the standard generator $\sigma_i$ of $B_{n}$ corresponds to the right-handed half Dehn twist along the real line segment $[i,i+1] \subset D_n$. 

Let $C_{n,m}$ be the unordered configuration space of $m$-points in $D_n$,
\[ C_{n,m}=\{ (z_{1},\ldots,z_{m}) \in D_{n}^{m} \: | \: z_{i}\neq z_{j} \textrm{ for } i \neq j\} \slash S_m .\]
Here $S_m$ denotes the symmetric group acting as a permutation of the indices.
Fix a sufficiently small $\varepsilon >0$, and let $d_{i}=(n+1)e^{(\frac{3}{2}+ i\varepsilon)\pi \sqrt{-1}}  \in \partial D_{n}$. We use $\mathbf{d} = \{d_{1},\ldots,d_{m}\}$ as a base point of $C_{n,m}$.
It is known that $H_{1}(C_{n,m}) = \Z^{n} \oplus \Z$, where the first $n$ components are spanned by the meridians of the hyperplanes $\{z_1=p_i\}$ $(i=1,\ldots, n)$, and the last component is spanned by the meridian of the discriminant $\bigcup_{1\leq i<j \leq n} \{z_{i}=z_{j}\}$. 

Let $\alpha$ be the homomorphism
\[
\xymatrix{ \alpha: \pi_{1}(C_{n,m}) \ar[rr]^-{\textrm{Hurewicz}} & &
H_{1}(C_{n,m};\Z) = \Z^{n} \oplus \Z \ar[r]^-{C} &\Z \oplus \Z = \langle x \rangle\oplus \langle d \rangle }
\]
where the first map is the Hurewicz homomorphism and the second map is defined by $C(x_1,\ldots,x_n,d) =( x_1+\cdots + x_n, d)$. Let $\pi: \widetilde{C_{n,m}} \rightarrow C_{n,m}$ be the covering that corresponds to $\Ker \alpha$. Take a point $\widetilde{\mathbf{d}} \in \pi^{-1}(\mathbf{d}) \subset \widetilde{C_{n,m}}$ and we use $\widetilde{\mathbf{d}}$ as a base point of $\widetilde{C_{n,m}}$.
By identifying $x$ and $d$ as deck translations,
$H_{m}(\widetilde{C_{n,m}};\Z)$ is a free $\Z[x^{\pm 1},d^{\pm1}]$-module of rank $\binom{n+m-2}{m}$. Moreover, since $\Ker \alpha$ is invariant under the $B_{n}$ action, $B_{n}$ acts on $H_{m}(\widetilde{C_{n,m}};\Z)$ as $\Z[x^{\pm 1},d^{\pm1}]$-module automorphisms.

Lawrence's representation is a variant of $H_{m}(\widetilde{C_{n,m}};\Z)$. It is a sub-representation $\mathcal{H}_{n,m}$ of $H^{lf}_{m}(\widetilde{C_{n,m}};\Z)$, the homology of locally finite chains.

Let $\mathsf{Y}$ be the $Y$-shaped graph with four vertices $c,r,v_1,v_2$ and oriented edges as shown in Figure \ref{fig:multiforks}(1).
A {\em fork} $F$ with the base point $d \in \partial D_{n}$ is an embedded image of $\mathsf{Y}$ into $D^2=\{z \in \C \: | \: |z| \leq n+1\}$ such that:
\begin{itemize}
\item All points of $\mathsf{Y}\setminus \{r,v_1,v_2\}$ are mapped to the interior of~$D_n$.
\item The vertex $r$ is mapped to $d$.
\item The other two external vertices $v_{1}$ and $v_{2}$ are mapped to the puncture points.
\end{itemize}
The image of the edge $[r,c]$, and the image of $[v_1,v_2]=[v_{1},c] \cup [c,v_{2}]$ viewed as a single oriented arc, are denoted by $H(F)$ and $T(F)$ respectively. We call $H(F)$ and $T(F)$ the \emph{handle} and the \emph{tine} of $F$.

\begin{figure}[htbp]
\centerline{\includegraphics[bb=123 618 480 718, width=120mm]{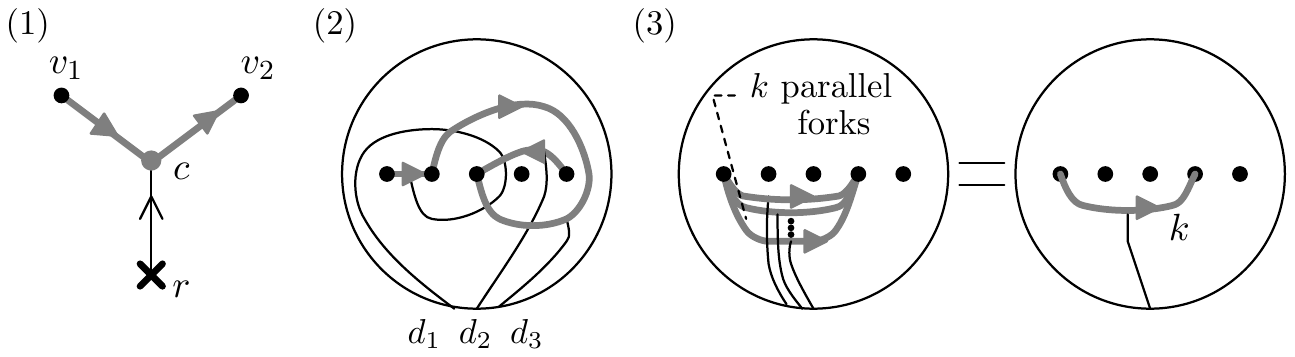}}
\caption{Forks and multiforks: to distinguish tines and handle, we often write tine of forks by a bold gray line. }
 \label{fig:multiforks}
\end{figure}

A {\em multifork} of dimension $m$ is an ordered tuples of $m$ forks $\bF= (F_{1},\ldots,F_{m})$ such that 
\begin{itemize}
\item $F_{i}$ is a fork based on $d_{i}$.
\item $T(F_{i}) \cap T(F_{j}) \cap D_{n} = \emptyset$ $(i\neq j)$.
\item $H(F_{i}) \cap H(F_{j}) = \emptyset$ $(i \neq j)$.
\end{itemize}
Figure \ref{fig:multiforks} (2) illustrates an example of a multifork of dimension $3$. We represent a multifork consisting of $k$ parallel forks by drawing a single fork labelled by $k$, as shown in Figure \ref{fig:multiforks} (3).

Let $\mathcal{E}_{n,m} =\{(e_{1},\ldots,e_{n-1}) \in \Z^{n-1}_{\geq 0} \: | \: e_1+\cdots +e_{n-1}=m \}$. For each $\be \in \mathcal{E}_{n,m}$, we assign a multifork $\bF_{\be}=\{F_{1},\ldots, F_{m}\}$ in Figure \ref{fig:standmultiforks}, which we call a {\em standard multifork}.
\begin{figure}[htbp]
\includegraphics[bb=191 614 388 710,width=65mm]{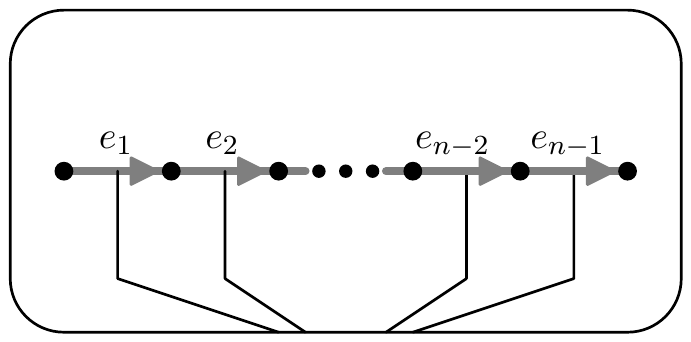}
\caption{Standard multifork $\bF_{\be}$ for $\be=(e_1,\ldots,e_{n-1}) \in \mathcal{E}_{n,m}$}
 \label{fig:standmultiforks}
\end{figure}

A multifork $\bF$ of dimension $m$ gives rise to a homology class $H_{m}^{lf}(\widetilde{C_{n,m}};\Z)$ in a following manner. Each handle $H(F_{i})$ of $\bF$ is seen as a path $\gamma_{i} \co [0,1] \rightarrow D_{n}$, so the handles of $\bF$ defines a path $H(\bF) = \{\gamma_{1},\ldots,\gamma_{m}\} \co [0,1] \rightarrow C_{n,m}$.
Let $\widetilde{H(\bF)} \co [0,1] \rightarrow \widetilde{C_{n,m}}$ be a lift of $H(\bF)$, taken so that $\widetilde{H(\bF)}(0)=\widetilde{\mathbf{d}}$. Let $\Sigma(\bF) = \left\{ \{z_{1},\ldots,z_{m} \} \in C_{n,m} \: | \: z_{i} \in T(F_{i}) \right\}$ and let $\widetilde{\Sigma}(\bF)$ be the connected component of $\pi^{-1}(\Sigma(\bF))$ containing $\widetilde{H(\bF)}(1)$. We equip an orientation of $\widetilde{\Sigma}(\bF)$ so that a canonical homeomorphism $T(F_1) \times \cdots \times T(F_m) \rightarrow \Sigma(\bF)$ is orientation preserving. Then $\widetilde{\Sigma}(\bF)$ represents an element of $H_{m}^{lf}(\widetilde{C_{n,m}};\Z)$.
By abuse of notation, we use $\bF$ to represent both multifork and the homology class $[\widetilde{\Sigma}(\bF)]$.

We graphically express relations among homology classes represented by multiforks by Figure \ref{fig:forkrule}, which we call \emph{fork rules}.

\begin{figure}[htbp]
\centerline{\includegraphics[bb=117 495 482 717, width=120mm]{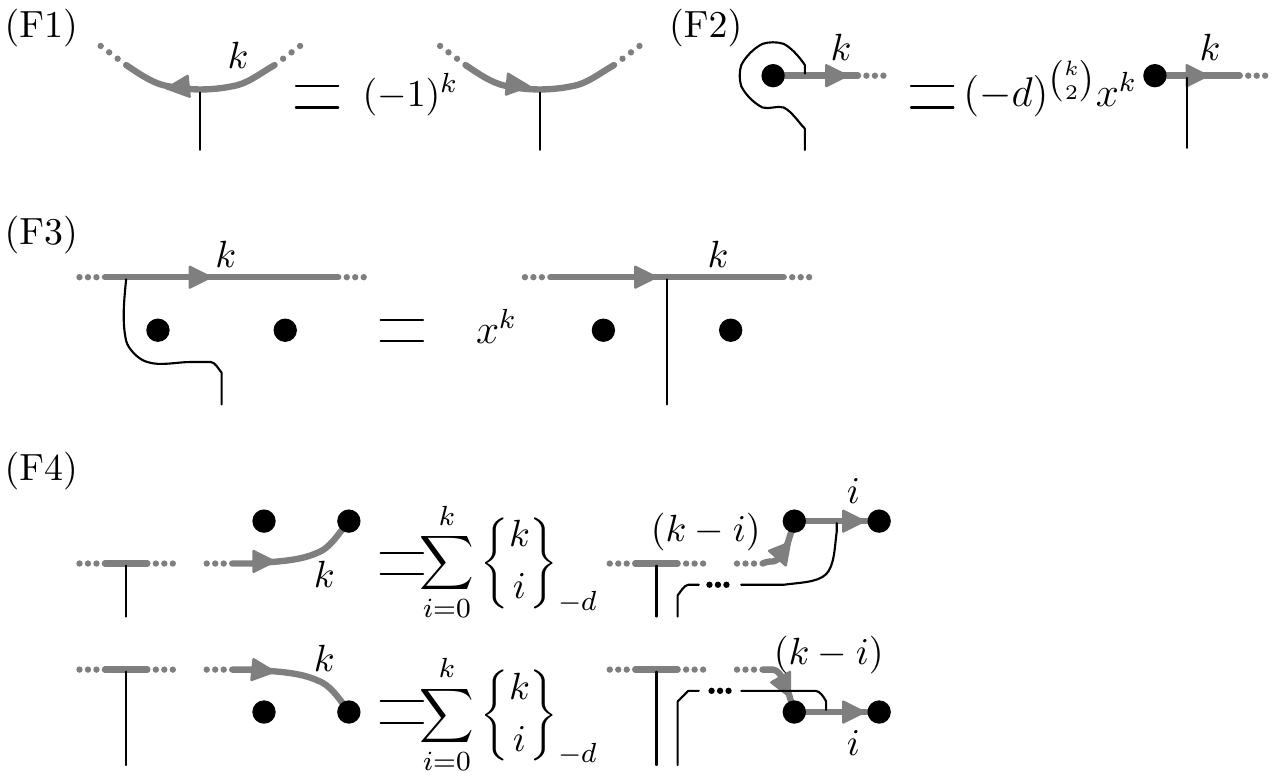}}
\caption{Geometric rewriting formula for multiforks. Here $\{a\}_{q} = \frac{q^{a}-1}{q-1}$ is a different version of $q$-integer, and ${a \brace b}_{q} = \frac{\{a\}_{q}!}{\{a-b\}_{q}!\{b\}_{q}!}$ is the version of $q$-binomial coefficient.}
 \label{fig:forkrule}
\end{figure}

Let $\mathcal{H}_{n,m}$ be the subspace of $H_{m}^{lf}(\widetilde{C_{n,m}};\Z)$ spanned by all multiforks. The set of standard multiforks forms a free basis of $\mathcal{H}_{n,m}$. $\mathcal{H}_{n,m}$ is invariant under the $B_{n}$ action, so we have the braid group representation  
\[ L_{n,m}: B_{n} \rightarrow \GL(\mathcal{H}_{n,m}) = \GL \left(\binom{n+m-2}{m};\Z[x^{\pm1},d^{\pm 1}] \right) \]
which we call \emph{Lawrence's representation}.

When $x$ and $d$ are put as fully generic complex numbers, all representations $\mathcal{H}_{n,m}$, $H^{lf}_{m}(\widetilde{C_{n,m}};\Z)$ and $H_{m}(\widetilde{C_{n,m}};\Z)$ are the same, but at non-generic parameters they might be different.


\begin{example}
Here we give a sample graphical calculation using the fork rules.
\begin{align*}
&\sigma_{1}\left(
\begin{minipage}{18mm}
\includegraphics*[width=18mm,bb=152 635 243 709]{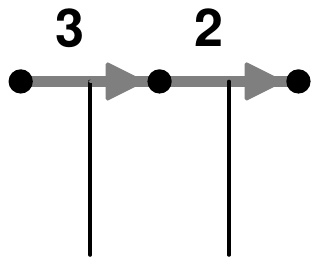}
\end{minipage}
\right)\\
& \qquad \quad = 
 \begin{minipage}{18mm}
\includegraphics*[width=18mm,bb=146 635 243 706]{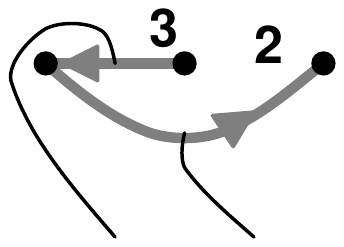}
\end{minipage} 
\overset{(F1, F2)}{ =}
(-1)^{3} (-d)^{3}x^{3} \begin{minipage}{18mm}
\includegraphics*[width=18mm,bb=152 635 243 709]{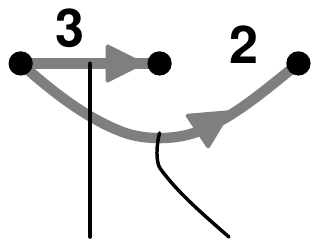}
\end{minipage} \\
& \qquad \quad \overset{(F4)}{=} d^{3}x^{3}\left( {2 \brace 0}_{-d}
\begin{minipage}{18mm}
\includegraphics*[width=18mm,bb=152 635 243 709]{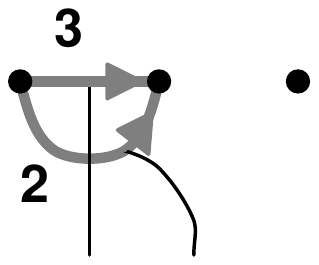}
\end{minipage} 
+ {2 \brace 1}_{-d}
\begin{minipage}{18mm}
\includegraphics*[width=18mm,bb=152 635 243 709]{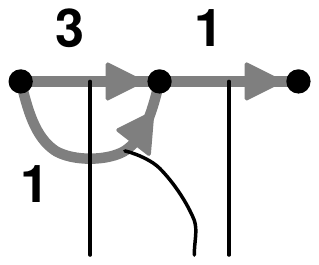}
\end{minipage} 
+ {2 \brace 2}_{-d}
\begin{minipage}{18mm}
\includegraphics*[width=18mm, bb=152 635 243 709]{zuF1.pdf}
\end{minipage} 
\right)\\
& \qquad \quad = d^{3}x^{3}\begin{minipage}{18mm}
\includegraphics*[width=18mm,bb=152 635 243 709]{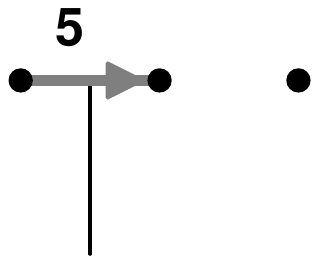}
\end{minipage} 
+ d^{3}(1-d)x^{3}\begin{minipage}{18mm}
\includegraphics*[width=18mm,bb=152 635 243 709]{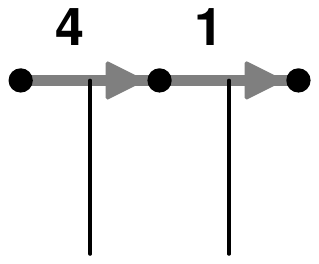}
\end{minipage} 
+ d^{3}x^{3}\begin{minipage}{18mm}
\includegraphics*[width=18mm, bb=152 635 243 709]{zuF1.pdf}.
\end{minipage}
\end{align*}
\end{example}

\subsection{Truncated Lawrence's representation}

We introduce a new braid group representation, defined in the case that the variable $-d$ is put as the $N$-th root of unity $\zeta_{N}^{2} = \exp(\frac{2 \pi \sqrt{-1}}{N})$. 

Let $\mathcal{H}^{\geq N}_{n,m}$ be the subspace of $\mathcal{H}_{n,m}$ spanned by $\{\bF_{\be}\: | \: \be \in\mathcal{E}_{n,m}^{\geq N}\}$, where  
\[ \mathcal{E}_{n,m}^{\geq N} = \{\be = (e_1,\ldots,e_{n-1}) \in \mathcal{E}_{n,m} \: | \: e_i \geq N \text{ for some } i\} \subset  \mathcal{E}_{n,m}.\]

We define $\overline{\mathcal{H}^{N}_{n,m}} = \mathcal{H}_{n,m} \slash \mathcal{H}^{\geq N}_{n,m}$. 
By abuse of notation, we use the same symbol $\bF$ to represent the image $\pi(\bF) \in \overline{\mathcal{H}^{N}_{n,m}}$ of the projection $\pi: \mathcal{H}_{n,m} \rightarrow \overline{\mathcal{H}^{N}_{n,m}}$. By definition, the set of standard multiforks without more than $(N-1)$ parallel tines $\{\bF_{\be}\}_{\be \in \overline{\mathcal{E}^{N}_{n,m}}}$, where
\[ \overline{\mathcal{E}^{N}_{n,m}} = \{\be = (e_1,\ldots,e_{n-1}) \in \mathcal{E}^{N}_{n,m}\: | \: e_{i} \leq N-1 \text{ for all } i\}  \subset  \mathcal{E}_{n,m}, \]
forms a basis of $\overline{\mathcal{H}^{N}_{n,m}}$. 
We denote the dimension of $\overline{\mathcal{H}^{N}_{n,m}}$, the cardinality of $\overline{\mathcal{E}^{N}_{n,m}}$ by $d^{N}_{n,m}$.

\begin{prop-definition}
At $d=-\zeta_{N}^{2}=-\exp(\frac{2 \pi \sqrt{-1}}{N})$, $\mathcal{H}^{\geq N}_{n,m}$ is a $B_{n}$-invariant subspace of $\mathcal{H}_{n,m}$, hence we have a linear representation
\[ l^{N}_{n,m}: B_{n} \rightarrow  \GL(\overline{\mathcal{H}^{N}_{n,m}}) = \GL(d^{N}_{n,m};\Z[x^{\pm 1}]). \]
We call $l^{N}_{n,m}$ \emph{truncated Lawrence's representation}.

\end{prop-definition}
\begin{proof}
For $\beta \in B_{n}$ and a standard multifork $\bF_{\be}$ with $\be \in \mathcal{E}_{n,m}^{\geq N}$, $\beta(\bF_{\be})$ is a multifork more than $(N-1)$ parallel tines. A crucial point is that for $k \geq N$, ${k \brace i}_{\zeta^{2}} = 0$ unless $i=0,k$. This observation and an iterated use of the fork rule (F4) show $\beta(\bF_{\be}) \in \mathcal{H}^{\geq N}_{n,m}$, as desired.
\end{proof}

\section{Truncated Lawrence's representation and quantum representation}
\label{sec:key}

In this section we explore the connection among the braid group representations we introduced in Section \ref{section:q} and \ref{sec:tLaw}.
We continue to assume that $\lambda$ is generic with respect to $q$.

Put $V'_{n,m} =\C v^{\lambda}_{0}\otimes V_{n-1,m} \subset V_{n,m}$. There is an isomorphism as $\C$-vector spaces $\Phi: V'_{n,m} \rightarrow W_{n,m}$ (for the precise definition of $\Phi$, which is not important here, see \cite{i2}).
Thus, by sending a natural basis of $V'_{n,m}$ we get a basis $\{w_{\be} = \Phi(v_{0}^{\lambda}\otimes(v_{e_1}^{\lambda}\otimes \cdots \otimes v_{e_{n-1}}^{\lambda}) \}_{\be=(e_1,\ldots,e_{n-1}) \in \mathcal{E}_{n,m}}$ of $W_{n,m}$, indexed by the same set $\mathcal{E}_{n,m}$ as the standard multiforks. 
Using this basis, we express the braid group representation $W_{n,m}$ as 
\[ \varphi^{W}_{n,m}:B_{n} \rightarrow \GL \left(\binom{n+m-2}{m};\C \right).\]
The next result says that $\varphi^{W}_{n,m}$ and $L_{n,m}$ are completely the same.

\begin{theorem}\cite[Theorem 6.1]{koh},\cite[Theorem 4.5]{i1}
\label{theorem:qish}
If $\lambda$ is generic with respect to $q$, then for an $n$-braid $\beta \in B_{n}$ we have an identity of matrices 
\begin{equation}
\label{eqn:qish}
\varphi^{W}_{n,m}(\beta) = L_{n,m}(\beta)|_{x=q^{-2\lambda},d=-q^{2}}. 
\end{equation}
\end{theorem}

We extend the identity (\ref{eqn:qish}) for the quantum representation $Y^{N}_{n,m}$ and truncated Lawrence's representation $l^{N}_{n,m}$. 
For $\be \in \overline{\mathcal{E}^{N}_{n,m}}$, we define $y_{\be} = \Phi(v^{\lambda}_{0}\otimes(v_{e_1}^{\lambda}\otimes \cdots \otimes v_{e_{n-1}}^{\lambda}))$. Then $y_{\be} \in Y^{N}_{n,m}$ and $\{y_{\be}\}_{ \overline{\mathcal{E}^{N}_{n,m}}}$ forms a basis of $Y^{N}_{n,m}$.
We express the braid group representation $Y_{n,m}$ by using this basis as
\[\varphi^{Y}_{n,m}:B_{n} \rightarrow \GL \left( d^{N}_{n,m};\C \right). \]

\begin{theorem}
\label{theorem:key} 
If $\lambda$ is generic with respect to $q$, for an $n$-braid $\beta \in B_{n}$ we have an identity of matrices 
\[ \varphi^{Y}_{n,m}(\beta) = l^{N}_{n,m}(\beta)|_{x=\zeta^{-2\lambda}, d=-\zeta^{2}} \]
\end{theorem}
\begin{proof}
We view $U_{N}(\lambda)$ as a quotient
$ V_{\lambda}|_{q=\zeta} \slash \{ v^{\lambda}_{i} = 0 \:| \:\text{for } i \geq N\},$ rather as a submodule of $V_{\lambda}|_{q=\zeta}$, and denote the quotient map by $p: V_{\lambda}|_{q=\zeta} \rightarrow U_{N}(\lambda)$. Then $Y^{N}_{n,m}$ is also seen as a quotient, $Y^{N}_{n,m} = p^{\otimes n}(W_{n,m}|_{q=\zeta})$ where $p^{\otimes n}: (V_{\lambda}|_{q=\zeta})^{\otimes n} \rightarrow U_{N}(\lambda)^{\otimes n}$ is the quotient map from $p$. Since $p$ is  an $U_{\zeta}(\sltwo)$-module homomorphism, $p^{\otimes n}$ is a surjection as a braid group representation.

By definition for the inclusion map $\iota : Y^{N}_{n,m} \hookrightarrow W_{n,m}|_{q=\zeta}$,  $p\circ\iota =\id$. Thus, $y_{\be} = p \circ \iota(y_{\be}) = p(w_{\be})$ for $\be \in \overline{\mathcal{E}^{N}_{n,m}}$. This observation, together with Theorem \ref{theorem:qish}, shows that the basis $\{y_{\be}\}_{\be \in \overline{\mathcal{E}^{N}_{n,m}}}$ of $Y^{N}_{n,m}$ corresponds to the standard multifork basis $\{\bF_{\be}\}_{\be \in \overline{\mathcal{E}^{N}_{n,m}}}$ of $\overline{\mathcal{H}^{N}_{n,m}}$, with substitution $x=\zeta^{-2\lambda}, d= -\zeta^{2}$.
\end{proof}

\section{Homological representation formula of the colored Alexander invariant}
\label{sec:formula}

Theorem \ref{theorem:key} philosophically provides a formula of the colored Alexander invariant, but rewriting the definition of colored Alexander invariant (\ref{eqn:ca}) in terms of $l^{N}_{n,m}(\beta) \cong \varphi^{Y}_{n,m}(\beta)$ requires several non-trivial computations. Our main task is to express the partial trace as a linear combination of the traces on intertwiner space.

To begin with, we compute the partial trace for $n=2$. 
Recall that $U_{N}(\lambda) \otimes U_{N}(\mu) = \bigoplus_{i=0}^{N-1} U_{N}(\lambda +\mu-2i) \otimes T_{i}(\lambda,\mu)$, where $T_i=T_{i}(\lambda,\mu)$ is the intertwiner space which is one-dimensional. Thus $f=f(\lambda,\mu) \in \End_{U_{\zeta}(\sltwo)}(U_{N}(\lambda) \otimes U_{N}(\mu))$ acts on each $T_{i}$ as a scalar multiple by $f_{i}  = f_{i}(\lambda,\mu)$.

\begin{lemma}
\label{lemma:coeff2}
The partial trace of $f=f(\lambda,\mu) \in \End_{U_{\zeta}(\sltwo)}(U_{N}(\lambda) \otimes U_{N}(\mu))$ is given by 
\begin{equation}
\label{eqn:formula0}
\trace_{2}(\id \otimes h)f = \left( \sum_{i=0}^{N-1} C_{i}^{N}(\lambda,\mu) f_{i} \right) \id. 
\end{equation}
Here $C^{N}_{i}(\lambda,\mu)$ satisfies the symmetry
\begin{equation}
\label{eqn:symmetry}
C_{i}^{N}(\lambda,\mu-2j) = C_{i+j}^{N}(\lambda,\mu),
\end{equation}
and is given by the formula
\begin{equation}
\label{eqn:coeff}
C^{N}_{i}(\lambda,\mu)= \frac{[\lambda;N-1]_{\zeta}}{[\lambda+\mu-2i;N-1]_{\zeta}}.
\end{equation}
\end{lemma}
\begin{proof}
Throughout the proof we adopt the convention that all indices are considered modulo $N$, unless otherwise specified.

First we prove the symmetry (\ref{eqn:symmetry}). It is sufficient to check (\ref{eqn:symmetry}) for the case $j=1$. A key observation is that 
\[ U_{N}(\lambda)\otimes U_{N}(\mu-2) = \bigoplus_{i=0}^{N-1} U_{N}(\lambda+\mu-2-2i) \cong \bigoplus_{i=0}^{N-1} U_{N}(\lambda+\mu-2i) = U_{N}(\lambda) \otimes U_{N}(\mu).
\]
This implies an isomorphism of intertwiner spaces
$ T_{i}(\lambda,\mu-2) \cong T_{i+1}(\lambda,\mu)$, hence $f_{i}(\lambda,\mu-2) = f_{i+1}(\lambda,\mu)$.
Also by definition of $h$,
\[ h(v_{i}^{\mu-2}) = \zeta^{(N-1)(-\mu+2+2i)} v^{\mu-2}_{i}, \quad h(v_{i+1}^{\mu}) = \zeta^{(N-1)(-\mu+2+2i)} v_{i+1}^{\mu}. \]
That is, the actions of $h$ on $T_{i}(\lambda,\mu-2)$ and $T_{i+1}(\lambda,\mu)$ are the same. Therefore we conclude
\begin{eqnarray*}
0 & = &\trace_{2}(\id \otimes h)f(\lambda,\mu-2)- \trace_{2}(\id \otimes h)f(\lambda,\mu)\\
 &= & \sum_{i=0}^{N-1} \left(C^{N}_{i}(\lambda,\mu-2) f_{i}(\lambda,\mu-2) - C^{N}_{i}(\lambda,\mu)f_{i}(\lambda,\mu)\right)\\
& = & \sum_{i=0}^{N-1} (C^{N}_{i+1}(\lambda,\mu-2) - C^{N}_{i}(\lambda,\mu))f_{i}(\lambda,\mu)
\end{eqnarray*}
This proves $C^{N}_{i+1}(\lambda,\mu-2) = C^{N}_{i}(\lambda,\mu)$.

Then we proceed to determine $C^{N}_{i}(\lambda,\mu)$. Thanks to (\ref{eqn:symmetry}), it is sufficient to determine $C^{N}_{0}(\lambda,\mu)$.

As a subspace, $T_{i}$ is identified with 
\[ T_{i} = \ker E \cap \textrm{span}\{v_{k}^{\lambda} \otimes v^{\mu}_{l}\:| \: k+l=i\}.\]
Here the condition on the indices $k+l=i$ is strict one, not regarded as modulo $N$. Take $t_{i} \in T_{i}$ so that $t_i= v_{0}^{\lambda} \otimes v_{i}^{\mu} + (\text{terms of }  v_{\geq 1}^{\lambda} \otimes v_{*}^{\mu})$. Then $\{F^{i-j} t_{j}\}_{j=0,\ldots,i}$ is a basis of the subspace  spanned by $\{v_{k}^{\lambda} \otimes v^{\mu}_{l}\:| \: k+l=i\}$ (again, the condition on the indices $k+l=i$ is strict one, not regarded as modulo $N$). Hence we can write
\[
v_{0}^{\lambda} \otimes v_{i}^{\mu} = \sum_{j=0}^{i} \alpha_{i,j}F^{i-j} t_{j}.
\]
for some $\alpha_{i,j}$. By applying $E$ to the both sides we get
\[ [\mu+1-i]_{\zeta} v_{0}^{\lambda} \otimes v_{i-1}^{\mu} = \sum_{j=0}^{N-1} \alpha_{i,j} EF^{i-j} t_{j}.\]
Since $[E,F^{i}] = [i]_{\zeta} F^{i-1}\frac{\zeta^{-(i-1)}K-\zeta^{(i-1)}K^{-1}}{\zeta-\zeta^{-1}}$ and $Et_{i}=0$ for $i\geq 1$,
\begin{eqnarray*}
[\mu+1-i]_{\zeta}v_{0}^{\lambda} \otimes v_{i-1}^{\mu} & = & \sum_{j=0}^{N-1} \alpha_{i,j} [i-j]_{\zeta} F^{(i-j)-1} \frac{\zeta^{-(i-j-1)}K-\zeta^{(i-j-1)}K^{-1}}{\zeta-\zeta^{-1}}t_{j} \\
& = & \sum_{j=0}^{N-1} \alpha_{i,j} [i-j]_{\zeta}[\lambda + \mu +1 - (i+j)]_{\zeta} F^{(i-1)-j}t_{j}.
\end{eqnarray*}
This shows
\[  \alpha_{i,j} = \frac{[\mu+1-i]_{\zeta}}{[i-j]_{\zeta}[\lambda + \mu +1 - (i+j)]_{\zeta}}\alpha_{i-1,j} . \]
By definition, $y_0 =  v_{0}^{\lambda} \otimes v_{0}^{\mu}$ so $\alpha_{0,0} = 1$. Therefore 
\begin{equation}
\label{eqn:alpha}
\alpha_{i,0} = \frac{[\mu;i]_{\zeta}}{[i]_{\zeta}![\lambda+\mu;i]_{\zeta}}.
\end{equation}

On the other hand, by using the basis $\{ v_{k}^{\lambda} \otimes v^{\mu}_{l} \}_{k+l=i}$, $F^{i-j}t_{j}$ is expressed as 
\[
F^{i-j}t_{j} = F^{i-j}(v_{0}^{\lambda} \otimes v_{j}^{\mu} + (\text{other terms}) ) = \zeta^{-(i-j)\lambda}[i;i-j]_{q}  v_{0}^{\lambda} \otimes  v_{i}^{\mu} + (\text{other terms}).
\]
Thus, the action of $(\id \otimes h)f$ is given by
\begin{eqnarray*}
(\id \otimes h)f(v_{0}^{\lambda}\otimes v_{i}^{\mu})\!& = & (\id \otimes h)f \left(\sum_{j=0}^{i} \alpha_{i,j}F^{i-j} t_{j} \right) = (\id \otimes h)\left(\sum_{j=0}^{i} \alpha_{i,j} f_{j}F^{i-j}t_{j}\right)\\
& = &
(\id \otimes h) \left\{\sum_{j=0}^{i} \alpha_{i,j} f_{j}\zeta^{-(i-j)\lambda}[i;i-j]_{\zeta} v_{0}^{\lambda} \otimes  v_{i}^{\mu}  + (\text{other terms}) \right\}\\
& = & \!\!\left(\sum_{j=0}^{i} \alpha_{i,j} f_{j}\zeta^{-(i-j)\lambda}\zeta^{-(N-1)(\mu-2i)}[i;i-j]_{\zeta}\right) v_{0}^{\lambda} \otimes  v_{i}^{\mu} + (\text{other terms}).
\end{eqnarray*}
Therefore
\begin{eqnarray*}
\trace_{2}( \id \otimes h)f & = & \left( \sum_{i=0}^{N-1}\sum_{j=0}^{i} \alpha_{i,j} f_{j}\zeta^{-(i-j)\lambda}\zeta^{-(N-1)\mu+2Ni-2i}[i;i-j]_{\zeta} \right) \id\\
& = & \left\{ \sum_{j=0}^{N-1}\left( \sum_{i=j}^{N-1} \alpha_{i,j} \zeta^{-(i-j)\lambda}\zeta^{-(N-1)\mu-2i}[i;i-j]_{\zeta}\right)  f_{j} \right\} \id.
\end{eqnarray*}
By (\ref{eqn:alpha}), we conclude
\[
C^{N}_{0}(\lambda,\mu) = \sum_{i=0}^{N-1} \alpha_{i,0} \zeta^{-i\lambda}\zeta^{-(N-1)\mu-2i}[i]_{\zeta}! = \sum_{i=0}^{N-1} \frac{[\mu;i]_{\zeta}}{[\lambda+\mu;i]_{\zeta}} \zeta^{-i\lambda}\zeta^{-(N-1)\mu-2i}. 
\]

The above formula of $C^{N}_{0}(\lambda,\mu)$ is simplified as follows.
Let us put $A_i=q^{-\lambda}[\mu-i]_{q}$, $B_{i}=q^{\mu}[\lambda-i]_{q}$, and $C_{i}=[\lambda+\mu-i]_{q}$. Then $C_{i} = q^{j}A_{i-j}+q^{-(i-j)}B_{j}$.

By induction on $N$, we show the identity 
\[
\sum_{i=0}^{N-1} \frac{A_0 A_1\cdots A_{i-1}}{C_0 C_1\cdots C_{i-1}} q^{-2i}= \frac{1}{C_0\cdots C_{N-2}}\sum_{i=1}^{N} q^{-(N-i)} {N \brack i}_{q} A_{0}\cdots A_{N-1-i} B_{0} \cdots B_{i-2}.
\]
Here we use a convention that for $X \in \{A,B,C\}$, $X_{0}X_{1}\cdots X_{j}=1$  if $j<0$). By using the inductive hypothesis for $N$, one computes
\begin{align*}
&\sum_{i=0}^{N} \frac{A_0 A_1\cdots A_{i-1}}{C_0 C_1\cdots C_{i-1}}q^{-2i} \\
& \quad = \frac{1}{C_0\cdots C_{N-1}} \left(\sum_{i=1}^{N} q^{-(N-i)} {N \brack i}_{q}\! A_{0}\cdots A_{N-1-i} B_{0} \cdots B_{i-2} C_{N-1} + q^{-2N}A_{0}\cdots A_{N-1}\right)\\
& \quad = \frac{1}{C_0\cdots C_{N-1}} \left(\sum_{i=1}^{N} q^{-(N-i)} {N \brack i}_{q} \!A_{0}\cdots A_{N-1-i} B_{0} \cdots B_{i-2} (q^{i-1} A_{N-i}+ q^{-N+i}B_{i-1}) \right. \\
& \qquad \qquad \qquad \qquad \qquad\qquad\qquad\qquad\qquad\qquad+ q^{-2N}{N \brack 0}_{q}A_{0}\cdots A_{N-1}\Biggr)\\
& \quad =  \frac{1}{C_0\cdots C_{N-1}} \left(\sum_{i=1}^{N+1}q^{-(N+1-i)} \left(  q^{i}{N \brack i}_{q} + q^{-N+i-1}{N \brack i-1}_{q} \right)\!A_{0}\cdots A_{N-i} B_{0} \cdots B_{i-2} \right) \\
& \quad = \frac{1}{C_0\cdots C_{N-1}}\sum_{i=1}^{N+1} q^{-(N+1-i)} {N+1\brack i}_{q} A_{0}\cdots A_{N-i} B_{0} \cdots B_{i-2},
\end{align*}
as desired. Since ${N\brack i}_{q}|_{q=\zeta}=0$ for $i=1,\ldots,N-1$, we conclude 
\[
C_{0}^{N}(\lambda,\mu)= \left.
\sum_{i=0}^{N-1} \frac{A_0 A_1\cdots A_{i-1}}{C_0 C_1\cdots C_{i-1}} q^{-2i}\right|_{q=\zeta} =\left. \zeta^{-(N-1)\mu}\frac{B_0 B_1\cdots B_{N-2}}{C_0 C_1\cdots C_{N-2}}\right|_{q=\zeta}=\frac{[\lambda;N-1]_{\zeta}}{[\lambda+\mu;N-1]_{\zeta}}. 
\]
\end{proof}

Next we compute the partial trace for general $n$. Recall that we have a splitting
\begin{equation}
\label{eqn:A3}
U_{N}(\lambda)^{\otimes n} = \bigoplus_{k=0}^{N-1} U_{N}(n\lambda -2k) \otimes T_{k}
\end{equation}
 where $T_{k}$ is the intertwiner space which $B_{n}$ acts on.
We denote the braid group action on $T_{k}$ by $\varphi_{T_k}:B_{n} \rightarrow \End(T_{k})$.

\begin{proposition}
\label{prop:ptrace}
\[
\trace_{2,\ldots,n}(\id \otimes h^{\otimes (n-1)}) f = \left(
\sum_{i=0}^{N-1} \frac{[\lambda;N-1]_{\zeta}}{[n\lambda-2i;N-1]_{\zeta}} \trace \varphi_{T_i}(\beta) \right) \id\]
\end{proposition}

\begin{proof}
As in the proof of Lemma \ref{lemma:coeff2}, we regard all indices are considered modulo $N$.
Let us view $U_{N}(\lambda^{\otimes n}) = U_{N}(\lambda) \otimes \left( U_{N}(\lambda)^{\otimes (n-1)}\right)$.
By (\ref{eqn:split}), as $U_{\zeta}(\sltwo)$-module we have a splitting
\[  U_{N}(\lambda)^{\otimes (n-1)} \cong \bigoplus_{i=0}^{N-1} U_{N}((n-1)\lambda-2i) \otimes T'_{i}\]
where $T'_i$ denotes the intertwiner space (the multiplicity of the direct summands). 
Hence
\begin{eqnarray}
\label{eqn:A0}
U_{N}(\lambda)\otimes U_{N}(\lambda)^{\otimes (n-1)}& = &\bigoplus_{i=0}^{N-1} \left( U_{N}(\lambda)\otimes U_{N}((n-1)\lambda-2i) \right) \otimes T'_{i}\\
\label{eqn:A1}
& = & \bigoplus_{i=0}^{N-1} \left\{ \bigoplus_{j=0}^{N-1} U_{N}(n\lambda-2(i+j))\otimes T'_{i}\right\}\\
\notag
\!\!\!(\text{put } k \equiv i+j \mod N)
& = &
\bigoplus_{i=0}^{N-1} \left\{ \bigoplus_{k=0}^{N-1} U_{N}(n\lambda-2k) \otimes T'_{i}\right\}\\
\label{eqn:A2}
& = & \bigoplus_{k=0}^{N-1}U_{N}(n\lambda -2k) \otimes \left( \bigoplus_{i=0}^{N-1} T'_{i} \right).
\end{eqnarray}

Let $\beta_{i,j} \in \End(U_{N}(n\lambda-2(i+j)) \otimes T'_{i})$ be the restriction of $\varphi_{U_N}(\beta)$ on $U_{N}(n\lambda-2(i+j))\otimes T'_{i}$ in the splitting (\ref{eqn:A1}). Comparing (\ref{eqn:A2}) with (\ref{eqn:A3}) we get $T_{i} \cong \bigoplus_{i=0}^{N-1} T'_{i}$ and 
\begin{equation}
\label{eqn:A4}
\sum_{i=0}^{N-1} \trace \beta_{i,k-i} = \trace \varphi_{T_k}(\beta).
\end{equation}

By (\ref{eqn:A0}), $\trace_{2,\ldots,n}$ in the left hand side of (\ref{eqn:A0}) is written in terms of the right hand side as
\begin{equation}
\label{eqn:A5}
\trace_{2,\ldots,n} = \sum_{i=0}^{N-1} \trace_{2} \otimes \trace_{T'_{i}} 
\end{equation}
where the first term is $\trace_{2}: \End(U_{N}(\lambda)\otimes U_{N}((n-1)\lambda-2i)) \rightarrow \End(U_{N}(\lambda)) $ and second term is  $\trace_{T'_{i}}: \End(T'_{i}) \rightarrow \C$.

By Lemma \ref{lemma:coeff2} and the observation (\ref{eqn:A5}), we conclude
\begin{eqnarray*}
\trace_{2,\ldots,n}(\id \otimes h^{\otimes (n-1)}) \varphi_{U_{N}}(\beta) &\!\!\!\stackrel{(\ref{eqn:formula0}), (\ref{eqn:A5})}{=}\! & \!\!\left\{ \sum_{i=0}^{N-1} \left( \sum_{j=0}^{N-1} C_{j}(\lambda, (n-1)\lambda -2i) \trace \beta_{i,j} \right) \right\} \id \\
& \stackrel{(\ref{eqn:symmetry})}{=} &\!\! \left\{ \sum_{i=0}^{N-1} \left( \sum_{j=0}^{N-1} C_{i+j}(\lambda, (n-1)\lambda) \trace \beta_{i,j} \right) \right\} \id \\
\!\!\!(\text{put } k \equiv i+j \mod N) & = &   \!\!
\left\{ \sum_{k=0}^{N-1} C_{k}(\lambda, (n-1)\lambda) \left(\sum_{i=0}^{N-1} \trace \beta_{i,k-i} \right) \right\} \id \\
& \stackrel{ (\ref{eqn:A4})}{=} &   \!\!
\left( \sum_{k=0}^{N-1} C_{k}(\lambda, (n-1)\lambda) \trace \varphi_{T_i}(\beta) \right) \id.
\end{eqnarray*}

\end{proof}

Finally we obtain a formula of colored Alexander invariant in terms of homological representations (truncated Lawrence's presentations) of the braid groups.

\begin{theorem}
\label{theorem:main}
Let $K$ be a knot represented as a closure of an $n$-braid $\beta$. Then 
\begin{equation}
\label{eqn:main}
\Phi^{N}_{K}(\lambda) = \zeta^{(N-1)\lambda e(\beta)}\sum_{i=0}^{N-1} \frac{[\lambda;N-1]_{\zeta}}{[n\lambda-2i;N-1]_{\zeta}}\left(\sum_{j=0}^{n-1} \left.\trace l^{N}_{n,i+(N-1)j}(\beta)\right|_{x=\zeta^{-2\lambda}, d=-\zeta^{2}}\right)
\end{equation}
\end{theorem}
\begin{proof}
As we have seen in (\ref{eqn:ttoy}), the intertwiner space $T_{i}$ is isomorphic to $\bigoplus_{j=0}^{n-1} Y_{n,i+(N-1)j}$ as a braid group representation. By Theorem \ref{theorem:key}, $\varphi^{Y}_{n,i+(N-1)j}(\beta)$ is equal to $l^{N}_{n,i+(N-1)j}(\beta)|_{x=\zeta^{-2\lambda}, d=-\zeta^{2}}$, hence Proposition \ref{prop:ptrace} gives the desired formula.
\end{proof}

Theorem \ref{theorem:main} recovers the classical formula of the Alexander polynomial.

\begin{example}
\label{exam:N=2}
Let us consider the case $N=2$ so we put $q = \zeta = \exp(\frac{2\pi\sqrt{-1}}{4})$. By Theorem \ref{theorem:main}, for a knot $K$ represented as the closure of an $n$-braid $\beta$, we have
\[ \Phi^{2}_{K}(\lambda) =  \zeta^{\lambda e(\beta)}\left( \frac{[\lambda]_{\zeta}}{[n\lambda]_{\zeta}}\sum_{j=0}^{n-1}\trace l^{2}_{n,2j}(\beta) - \frac{[\lambda]_{\zeta}}{[n\lambda]_{\zeta}} \sum_{j=0}^{n-1}\trace l^{2}_{n,2j+1}(\beta) \right).\]

A crucial observation is that we have an isomorphism of the braid group representation
\begin{equation}
\label{eqn:tl}  
\overline{\mathcal{H}^{2}_{n,k}} = \bigwedge^{k}\overline{\mathcal{H}^{2}_{n,1}} =  \bigwedge^{k}\mathcal{H}_{n,1} , \text{ hence } 
\bigoplus_{k=0}^{\infty} \overline{\mathcal{H}^{2}_{n,k}} = \bigwedge \overline{\mathcal{H}^{2}_{n,1}} = \bigwedge \mathcal{H}_{n,1}
\end{equation}
That is, truncated Lawrence's representation is identified with the exterior powers of the reduced Burau representation. This shows that $
l^{2}_{n,k}(\beta) = \bigwedge^{k}l^{2}_{n,1}(\beta) = \bigwedge^{k}L_{n,1}(\beta) $ so
\[ \sum_{j=0}^{n-1}\trace l^{2}_{n,2j}(\beta) = \trace \bigwedge^{\sf even} L_{n,1}(\beta), \ \sum_{j=0}^{n-1}\trace l^{2}_{n,2j+1}(\beta) = \trace \bigwedge^{\sf odd} L_{n,1}(\beta). \]
Here $\bigwedge^{\sf even}$ and $\bigwedge^{\sf odd}$ denote the even and the odd degree part of the exterior powers. Hence 
\begin{eqnarray*}
\Phi^{2}_{K}(\lambda)& =& \zeta^{\lambda e(\beta)} \frac{\zeta^{\lambda}-\zeta^{-\lambda}}{\zeta^{n\lambda} - \zeta^{-n\lambda}}\left( \trace \bigwedge^{\sf even}L_{n,1}(\beta)|_{x=\zeta^{-2\lambda}}  - \trace \bigwedge^{\sf odd}L_{n,1}(\beta)|_{x=\zeta^{-2\lambda}} \right) \\
& = & \zeta^{\lambda e(\beta)} \frac{\zeta^{\lambda}-\zeta^{-\lambda}}{\zeta^{n\lambda} - \zeta^{-n\lambda}} \det(I-L_{n,1}(\beta)|_{x=\zeta^{-2\lambda}}),
\end{eqnarray*}
and by rewriting in terms of $x=\zeta^{-2\lambda}$, we get
\[ CA_{K}^{2}(x) = \Phi^{2}_{K}(\lambda)|_{x=\zeta^{-2\lambda}} = x^{-\half e(\beta)}\frac{x^{\half}-x^{-\half}}{x^{\frac{n}{2}} - x^{-\frac{n}{2}}} \det(I-L_{n,1}(\beta)).\]
This is nothing but the well-known formula of the Alexander polynomial \cite{bir}.
\end{example}

Unfortunately, for $N>2$, we have no clear topological interpretations or simplifications of the formula (\ref{eqn:main}). It is an interesting and important problem to understand the formula (\ref{eqn:main}) in terms of the topology of the knot complements.
In particular, it is desirable to give an independent and direct proof of the fact that (\ref{eqn:main}) yields a knot invariant. One of a difficulty toward the better understanding of the formula (\ref{eqn:main}) is a lack of our understanding of $l^{N}_{n,m}$. In a light of (\ref{eqn:tl}), we expect that $l^{N}_{n,m}$ is deduced from $\bigoplus_{j=1}^{N-1} \overline{\mathcal{H}^{N}_{n,j}} = \bigoplus_{j=1}^{N-1} \mathcal{H}_{n,j}$.

\end{document}